\documentclass[12pt]{amsart}
\usepackage{amssymb,amscd}
\usepackage{verbatim}

\usepackage{amsmath,amssymb,graphicx,mathrsfs}   
\usepackage[colorlinks=true,allcolors = blue]{hyperref} 

\let\frak\mathfrak

\def\>{\relax\ifmmode\mskip.666667\thinmuskip\relax\else\kern.111111em\fi}
\def\<{\relax\ifmmode\mskip-.333333\thinmuskip\relax\else\kern-.0555556em\fi}
\def\vsk#1>{\vskip#1\baselineskip}
\def\vv#1>{\vadjust{\vsk#1>}\ignorespaces}
\def\vvn#1>{\vadjust{\nobreak\vsk#1>\nobreak}\ignorespaces}

  \let\ssize\scriptstyle
\let\sssize\scriptscriptstyle

\let\Medskip\medskip
\def\medskip{\par\Medskip}
\let\Bigskip\bigskip
\def\bigskip{\par\Bigskip}

\let\Maketitle\maketitle
\def\maketitle{\Maketitle\thispagestyle{empty}\let\maketitle\empty}

\newtheorem{thm}{Theorem}[section]
\newtheorem{cor}[thm]{Corollary}
\newtheorem{lem}[thm]{Lemma}

\theoremstyle{definition}                                  
\newtheorem{exmp}{Example}[section]

\numberwithin{equation}{section}

\theoremstyle{definition}
\newtheorem*{rem}{Remark}

\let\mc\mathcal
\let\nc\newcommand

\let\al\alpha
\let\bt\beta

\let\la\lambda

\let\phi\varphi
\let\si\sigma

\let\Om\Omega

\let\der\partial

\let\ox\otimes

\let\geq\geqslant

\let\leq\leqslant

\let\on\operatorname
\let\bi\bibitem
\let\bs\boldsymbol

\def\C{{\mathbb C}}
\def\Z{{\mathbb Z}}

\def\B{{\mc B}}
\def\F{{\mc F}}

\def\+#1{^{\{#1\}}}

\def\End{\on{End}}

\def\Wr{\on{Wr}}

\def\beq{\begin{equation}}
\def\eeq{\end{equation}}
\def\be{\begin{equation*}}
\def\ee{\end{equation*}}

\nc{\bea}{\begin{eqnarray*}}
\nc{\eea}{\end{eqnarray*}}
\nc{\bean}{\begin{eqnarray}}
\nc{\eean}{\end{eqnarray}}
\nc{\Ref}[1]{{\rm(\ref{#1})}}

\def\g{{\mathfrak g}}

\nc{\Il}{{\mc I_{\bs\la}}}
\nc{\bla}{{\bs\la}}
\nc{\Fla}{\F_\bla}
\nc{\tfl}{{T^*\Fla}}
\nc{\GL}{{GL_n(\C)}}
\nc{\GLC}{{GL_n(\C)\times\C^*}}

\let\sd s 

\def\ddk_#1{\kk_{#1}\<\>\frac\der{\der\<\>\kk_{#1}}}

\def\bul{\mathbin{\raise.2ex\hbox{$\sssize\bullet$}}}
\def\intt{\mathchoice
{\mathop{\raise.2ex\rlap{$\,\,\ssize\backslash$}{\intop}}\nolimits}
{\mathop{\raise.3ex\rlap{$\,\sssize\backslash$}{\intop}}\nolimits}
{\mathop{\raise.1ex\rlap{$\sssize\>\backslash$}{\intop}}\nolimits}
{\mathop{\rlap{$\sssize\<\>\backslash$}{\intop}}\nolimits}}

\let\kk q 
\let\cc c

\let\Ko K

\def\GZ/{Gelfand-Zetlin}
\def\KZ/{{\slshape KZ\/}}
\def\qKZ/{{\slshape qKZ\/}}
\def\XXX/{{\slshape XXX\/}}

\def\Sym{\on{Sym}}

\nc{\A}{{\mc C}}

\def\FF{{\mathbb F}}
\def\Sing{{\on{Sing}}}
\def\sll{{\frak{sl}}}

\def\slt{{\frak{sl}_2}}
\def\Ant{{\on{Ant}}}
\def\Ik{{\mc I_k}}

\def\A{{\mathbb A}}
\def\AA{{\mc A}}

\begin{document}

\hrule width0pt
\vsk->

\title[Remarks on the Gaudin model  modulo  $p$]
{Remarks on the Gaudin model  modulo  $p$}

\author
[ Alexander Varchenko]
{Alexander Varchenko$\>^\star$}

\maketitle

\begin{center}
{\it $^{\star}\<$Department of Mathematics, University
of North Carolina at Chapel Hill\\ Chapel Hill, NC 27599-3250, USA\/}

\end{center}

{\let\thefootnote\relax
\footnotetext{\vsk-.8>\noindent
$^\star\<${\sl E\>-mail}:\enspace anv@email.unc.edu\>,
supported in part by NSF grant  DMS-1665239}}

\begin{abstract}
We discuss the Bethe ansatz in the Gaudin model on the tensor product of finite-dimensional
$\slt$-modules over the field $\FF_p$ with $p$ elements, where $p$ is a prime number.
We define the Bethe ansatz equations and show that if $(t^0_1,\dots,t^0_k)$ is a 
solution of the Bethe ansatz equations, then the corresponding 
Bethe vector is an eigenvector of the Gaudin Hamiltonians. 
We characterize solutions $(t^0_1,\dots,t^0_k)$ of the Bethe ansatz equations
as certain  two-dimensional subspaces of the space
of polynomials $\FF_p[x]$. We consider the case
when the number of parameters $k$ equals 1. In that case we show that the Bethe algebra,
generated by the Gaudin Hamiltonians, is isomorphic to the algebra of functions on the 
scheme defined by the Bethe ansatz equation. If $k=1$ and in addition the tensor product 
is the product of vector representations, then the Bethe algebra is also
 isomorphic to the algebra of functions on the fiber of a suitable Wronski map.

\end{abstract}

\bigskip

\hfill In memory of Egbert Brieskorn (1936--2013)

{\small \tableofcontents  }

\setcounter{footnote}{0}
\renewcommand{\thefootnote}{\arabic{footnote}}

\section{Introduction}
The Gaudin model is a certain collection of commuting linear operators on the tensor product $V=\ox_{i=1}^nV_i$
of representations of a Lie algebra $\g$.  The operators are called the Gaudin Hamiltonians. 
The Bethe ansatz is a method used to construct common eigenvectors and eigenvalues of the Gaudin
operators. 
One looks for an eigenvector in a certain form $W(t)$, where  $W(t)$ is a $V$-valued function
of some parameters $t=(t_1,\dots,t_k)$. One introduces a system of equations on the parameters, called
the Bethe ansatz equations, and shows that if $t^0$ is a solution of the system, then the vector $W(t^0)$ is
an eigenvector of the Gaudin Hamiltonians, see for example 
\cite{B,G,FFR,MV1,MV2,MTV1,MTV4,RV,SchV,SV1, V1,V2,V3}.
The Gaudin model has strong relations with the Schubert calculus  and real algebraic geometry, see
for example  \cite{MTV2, MTV3, So}.  

All that is known in the case when the Lie algebra  $\g$ is defined
over the field $\C$ of complex numbers. 
In the paper we consider the case of the field $\FF_p$ with $p$ elements, where $p$ is a prime number, cf.
\cite{SV3}.  We carry out 
the first steps of the Bethe ansatz, the deeper parts of the Gaudin model over a 
finite field  remain to be developed.
We consider the case of the Lie algeba $\slt$, where the notations 
and constructions  are shorter and simpler. 

It is known that over $\C$, the Gaudin model is a semi-classical  limit of the KZ differential 
equations of conformal field theory, and the construction
of the multidimensional hypergeometric solutions of the KZ differential equations lead, in that limit,
 to the Bethe ansatz construction of  eigenvectors of the Gaudin Hamiltonians,  see \cite{RV}. The $\FF_p$-analogs of the hypergeometric solutions of the KZ differential equations were constructed recently in \cite{SV3}, see also  \cite{V5}.  Thus the 
constructions of this paper may be thought of as
a semi-classical limit of the constructions in \cite{SV3}.

In Section \ref{2} we define the Bethe ansatz equations and show that if $(t^0_1,\dots,t^0_k)$ is a 
solution of the Bethe ansatz equations, then the corresponding 
Bethe vector is an eigenvector of the Gaudin Hamiltonians.
In Section \ref{TD} we characterize solutions $(t^0_1,\dots,t^0_k)$ of the Bethe ansatz equations
as certain  two-dimensional subspaces of the space of polynomials $\FF_p[x]$.
In Section \ref{sec ex} we consider the case  in which the number $k$ of the
parameters equals 1. In that case
we show that the Bethe algebra,
generated by the Gaudin Hamiltonians, is isomorphic to the algebra of functions on the 
scheme defined by the Bethe ansatz equation, see Theorem \ref{thm AB}.  
If $k=1$ and in addition the tensor product is the product of vector representations, 
then the Bethe algebra is also isomorphic to the algebra of functions on the fiber of a suitable Wronski map, 
see  Corollary
\ref{cor ABC}.

\medskip
The author thanks  W.\,Zudilin for useful discussions
 and the Hausdorff Institute for Mathematics in Bonn for hospitality in May-July of 2017.

\section{$\sll_2$ Gaudin model}
\label{2}

\subsection{$\slt$ Gaudin model over $\C$}
Let $e,f,h$ be the
standard basis of the  complex Lie algeba $\slt$ with $[e,f]=h$, $[h,e]=2e$, $[h,f]=-2f$.
The element  
\bean
\label{Casimir}
\Omega =  e \otimes f + f \otimes e +
                       \frac{1}{2} h \otimes h\  \in\  \slt \ox\slt
                       \eean
 is called the Casimir element.  Given $n$, for $1\leq i<j\leq n$ let $\Om^{(i,j)} \in (U(\slt))^{\ox n}$
 be the element equal to
 $\Omega$ in the $i$-th and $j$-th factors and to 1 in other factors.
 Let $z^0=(z^0_1,\dots,z^0_n)\in\C^n$ have distinct coordinates.
 For  $s=1,\dots,n$  introduce
 \bean
 \label{Ham}
 H_s(z^0)= \sum_{l\ne s}\frac{\Om^{(s,l)}}{z^0_s-z^0_l} \ \in \ (U(\slt))^{\ox n},
 \eean
 the {\it Gaudin Hamiltonians}, see \cite{G}.
 For  any $s,l$, we have
 \bean
 \label{Flat}
 \left[H_s (z^0), H_l (z^0) \right]
=0,
\eean 
and for any $x\in\slt$ and $s$ we have
\bean
\label{h invar}
[H_s(z^0), x\ox 1\ox\dots\ox 1+\dots + 1\ox\dots\ox 1\ox x] =0.
\eean

Let $V= \ox^n_{i=1}V_i$ be a tensor product of $\slt$-modules.  The 
commutative
subalgebra of $\End(V)$ generated by the Gaudin Hamiltonians $H_i(z^0)$, $i=1,\dots,n$, and the identity operator
$\on{Id}$ is called the  {\it Bethe algebra} of $V$.
If $W\subset V$ is a subspace invariant with respect to the Bethe algebra, then the restriction 
of the Bethe algebra  to $W$ is called the Bethe algebra of $W$, denoted by $\B(W)$.

The general problem is to describe the Bethe algebra, its common eigenvectors and eigenvalues.

\subsection{Irreducible $\slt$-modules}
For a nonnegative integer $i$ denote by $L_i$ the irreducible $i+1$-dimensional module with
basis $v_i, fv_i,\dots,f^iv_i$ and action $h.f^kv_i=(i-2k)f^kv_i$ for $k=0,\dots,i$; $f.f^kv_i=f^{k+1}v_i$ for $k=0,\dots, i-1$,
$f.f^iv_i=0$; $e.v_i=0$, $e.f^kv_i=k(i-k+1)f^{k-1}v_i$ for $k=1,\dots,i$.

For $m = (m_1,\dots,m_n) \in \Z^n_{\geq 0}$, denote
$|m| = m_1 + \dots + m_n$ and
$L^{\otimes m} =  L_{m_1}  \otimes  \dots  \otimes L_{m_n}$.  
 For
$J=(j_1,\dots,j_n) \in \mathbb{Z}_{\geq 0}^n$, with $j_s\leq m_s$ for $s=1,\dots,n$, the vectors
\bean
\label{fv}
f_Jv_m := f^{j_1}v_{m_1}\otimes \dots \otimes f^{j_n}v_{m_n} 
\eean
form a basis of
$L^{\otimes m}$. We have
\bea
f.f_Jv_m &=& \sum_{s=1}^n f_{J+1_s}v_m,
\qquad
h.f_Jv_m = ( |m|-2|J|) f_Jv_m,
\\
&&
e.f_Jv_m = \sum_{s=1}^n j_s (m_s-j_s+1) f_{J-1_s}v_m.
\eea
For $\lambda \in \Z$, introduce the weight subspace $L^{\otimes m}[\lambda] = \{ \ v \in L^{\otimes m} \
| \ h.v = \lambda v \}$ and the singular weight subspace
$\Sing L^{\otimes m}
[\lambda] = \{ \ v \in L^{\otimes m}[\lambda] \ | \ h.v = \lambda v, \
e.v = 0 \} $. We have the weight decomposition
$L^{\ox m} = \oplus_{k=0}^{|m|}L^{\ox m}[|m|-2k]$.
Denote
\bea
\Ik =\{ J\in \Z^n_{\geq 0}\ | \ |J|=k, \,j_s\leq m_s,\ s=1,\dots,n\}.
\eea
The vectors $(f_Jv)_{J\in\Ik}$ form a basis of $L^{\otimes m}[|m|-2k]$.

\smallskip

The Bethe algebra $\B(L^{\ox m})$ preserves each of the subspaces $L^{\otimes m}[|m|-2k]$
and $\Sing L^{\otimes m}[|m|-2k]$ by \Ref{h invar}.  If $w\in L^{\ox m}$ is a common eigenvector of the Bethe algebra,
 then for any $x\in\slt$ the vector  
$x.w$ is also an eigenvector with the same eigenvalues. 
These  observations
show that in order
to describe  $\B(L^{\ox m})$, its eigenvectors and eigenvalues it is enough  
to describe for all $k$
the algebra  $\B(\Sing L^{\ox m}[|m|-2k])$, its eigenvectors and eigenvalues.

\subsection
{Bethe ansatz on 
 $\Sing L^{\otimes m} \big[ |m| - 2k \big]$ over $\C$}
  \label{Sol in C}

Given $k, n \in \Z_{>0}$, $m = ( m_1, \dots , m_n) \in \Z_{> 0}^n$.
Let $z^0=(z^0_1,\dots,z^0_n)\in\C^n$ have distinct coordinates.
The system of the {\it Bethe ansatz equations}  is the system of equations
\bean
\label{BAE}
\sum_{j\ne i}\frac 2{t_i-t_j} -\sum_{s=1}^n\frac{m_s}{t_i-z^0_s}=0,
\qquad i=1,\dots,k,
\eean
on $t=(t_1,\dots,t_k)\in\C^k$.  If $(t_1^0,\dots,t_k^0,z^0_1,\dots,z_n^0)\in\C_p^{k+n}$ has distinct coordinates, denote
\bean
\label{la-i}
\la_s(t^0,z^0)  = \sum_{l\ne s} \frac {m_{s}m_l/2}{z^0_s-z^0_l}
-\sum_{i=1}^k \frac{m_s}{z^0_s-t^0_i},
\qquad s=1,\dots,n.
\eean

For any function or differential form $F(t_1, \dots , t_k)$, denote
\bea
\Sym_t  [F(t_1, \dots , t_k)]  = \sum_{\sigma \in S_k}\!
F(t_{\sigma_1}, \dots , t_{\sigma_k}) ,
\ \ 
\Ant_t  [F(t_1, \dots , t_k)]  =  \sum_{\sigma \in S_k}\!
(-1)^{|\sigma|}F(t_{\sigma_1}, \dots , t_{\sigma_k}) .
\eea
For $J=(j_1,\dots,j_n)\in \Ik$ define the {\it weight function}
\bean
\label{WJ}
W_J(t,z)  = \frac{1}{j_1! \dots j_n!} 
                            \Sym_t \left[ \prod_{s=1}^{n}
                            \prod_{i=1}^{j_s}
                            \frac{1}{t_{ j_1 + \dots + j_{s-1}+i} - z_s}
                           \right]  .
\eean
For example,
\bea
&&
W_{(1,0,\dots,0)} =  \frac{1}{t_1 - z_1} ,
\qquad
W_{(2,0,\dots,0)}  =  \frac{1}{t_1 - z_1}  \frac{1}{t_2 - z_1} ,
\\
&&
\phantom{aaa}
W_{(1,1,0,\dots,0)} =   \frac{1}{t_1-z_1} \frac{1}{t_2-z_2}
+ \frac{1}{t_2-z_1}  \frac{1}{t_1-z_2}  .
\eea
The function
\bean
\label{vW}
W_{k,n,m}(t,z)=\sum_{J\in \Ik}W_J(t,z) f_Jv_m
\eean
is  the  $L^{\otimes m}[ |m| -2k]$-valued {\it vector weight function}.

\begin{thm} 
[\cite{RV,B}, cf. \cite{SV1}]
\label{thm BA}
If $(t^0,z^0)=(t_1^0,\dots,t_k^0,z^0_1,\dots,z_n^0)$ is a solution of the Bethe ansatz equations
\Ref{BAE}, then
the vector $W_{k,n,m}(t^0,z^0)$ lies in $\Sing L^{\ox m}[|m|-2k]$ and is an eigenvector
of the Gaudin Hamiltonians, moreover,
\bean
\label{eig}
H_i(z^0).W_{k,n,m}(t^0,z^0) = \la_i(t^0,z^0) W_{k,n,m}(t^0,z^0),\qquad i=1,\dots, n.
\eean

\end{thm}

The eigenvector $W_{k,n,m}(t^0,z^0)$ is called the {\it Bethe eigenvector}. On the Bethe 
eigenvectors see, for example,
\cite{SchV, MV1,MV2,V1,V2,V3}.

The fact that  $W_{k,n,m}(t^0,z^0)$ in Theorem \ref{thm BA} lies
in $ \Sing L^{\otimes m} [ |m|-2k ]$ may be reformulated as follows.
For any $J\in \mc I_{k-1}$,  we have
\bean
\label{Rel}
\sum_{s=1}^{n} (j_s + 1) (m_s - j_s) W_{J+ {\bf 1}_s}(t^0,z^0)   =  0 ,
\eean
where we set $ W_{J+ {\bf 1}_s} (t^0,z^0)=0$ if  $J+ {\bf 1}_s \notin \Ik$.

\subsection{Proof of Theorem \ref{thm BA}}
We sketch the proof following \cite{SV1}. The intermediate statements in this proof will be used later  when constructing eigenvectors of the Bethe algebra  over $\FF_p$.  The proof is based on 
the following cohomological relations.

Given $k, n \in \Z_{>0}$ and a multi-index $J = (j_1, \dots , j_n)$ with $|J| \leq k$,
introduce a differential form
\bea
&&
\eta_J  =  \frac{1}{j_1 ! \cdots j_n !}
\operatorname{Ant}_t
\Big[
\frac{d(t_1 - z_1)}{t_1 -z_1} \wedge
\dots
\wedge \frac{d(t_{j_1} - z_1)}{t_{j_1} - z_1} \wedge
\frac{d(t_{j_1+1} - z_2)}{t_{j_1+1} - z_2} \wedge \dots
\\
&&
\phantom{aaaaasssaaaaaaaa}
 \wedge 
\frac{d(t_{j_1+ \dots + j_{n-1}+1} - z_n)}{t_{j_1+ \dots + j_{n-1}+1} - z_n}
\wedge \dots \wedge
\frac{d(t_{j_1+ \dots + j_{n}} - z_n)}{t_{j_1+ \dots + j_{n}} - z_n}
\Big]  ,
\eea
which is  a logarithmic differential form on $\C^n\times \C^k$
with coordinates $z,t$.
If $|J| = k$, then for any $z^0\in \C^n$ we have on $\{z^0\} \times \C^k$ the identity
\bean
\label{ew}
\eta_J\big|_{\{z^0\}\times\C^k}  = W_J(t, z^0) dt_1  \wedge  \dots \wedge dt_k .
\eean
\begin{exmp}
For $k = n = 2$ we have
\bea
\eta_{(2,0)} \
& =&
\ \frac{d(t_1 - z_1)}{t_1 -z_1} \wedge \frac{d(t_2 - z_1)}{t_2 -z_1} ,
\\
\eta_{(1,1)} \
& = &
\ \frac{d(t_1 - z_1)}{t_1 -z_1} \wedge \frac{d(t_2 - z_2)}{t_2 -z_2} -
\frac{d(t_2 - z_1)}{t_2 -z_1} \wedge \frac{d(t_1 - z_2)}{t_1 -z_2} .
\eea
\end{exmp}

Introduce the logarithmic differential 1-forms
\bea
\alpha 
& =  &
\sum_{1\leq i<j\leq n}
 \frac{m_i m_j}{2}
 \frac{d (z_i - z_j)}{z_i -z_j}  + 
\sum_{1\leq i<j\leq k} 2
\frac{d (t_i - t_j)}{t_i -t_j} -
\sum_{s=1}^{ n} \sum_{i=1}^k m_s
 \frac{d (t_i - z_s)}{t_i -z_s} ,
\\
\alpha'
& = & 
\sum_{1\leq i<j\leq k} 2
\frac{d (t_i - t_j)}{t_i -t_j} -
\sum_{s=1}^{ n} \sum_{i=1}^k m_s
 \frac{d (t_i - z_s)}{t_i -z_s} .
\eea
We shall use the following algebraic identities for logarithmic differential forms.

\begin{thm} [\cite{SV1}]
\label{ss2}
We have
\bean
\label{id1}
\alpha' \wedge \eta_J
= \sum_{s=1}^n (j_s +1) (m_s - j_s) 
 \eta_{J + {\bf 1}_s},  
\eean
for any $J$ with $|J|=k-1$, and
\bean
\label{id2}
\alpha \wedge   \sum_{J\in\Ik}  \eta_J f_Jv_m
    = \sum_{i<j} \Omega^{(i,j)} 
    \frac{d(z_i-z_j)}{z_i-z_j}  \wedge 
   \sum_{|J|=k}  \eta_J  f_Jv_m .
\eean
\end{thm}

\begin{proof}  Identity \Ref{id1} is Theorem 6.16.2 in \cite{SV1} for the case of the Lie algebra $\slt$.
Identity \Ref{id2} is Theorem 7.5.2'' in \cite{SV1} for the case of the Lie algebra $\slt$.
\end{proof}

If $(t^0,z^0)$ is a solution
of the Bethe ansatz equations,
 then $\al'|_{(t^0,z^0)}=0$ and formulas \Ref{id1}, 
\Ref{ew}  give \Ref{Rel}. Similarly, if 
$(t^0,z^0)$ is a solution
of the Bethe ansatz equations,
  then $\al|_{(t^0,z^0)}=\sum_{s=1}^n \la_s(t^0,z^0)dz_s$ 
and formulas \Ref{id2} and \Ref{ew} give \Ref{eig}. Theorem \ref{thm BA} is proved.

\subsection
{Bethe ansatz on 
 $\Sing L^{\otimes m} \big[ |m| - 2k \big]$ over $\FF_p$}
  \label{Sol in F}

Given $k, n \in \Z_{>0}$,  $m = ( m_1, \dots , m_n) \in \Z_{>0}^n$,  
let $p$ be a prime number. Consider the Lie algebra $\slt$
as an algebra  over
the field $\FF_p$ and the $\slt$-modules $L_{m_s}$, $s=1,\dots,n$, over $\FF_p$. 
Let $z^0=(z_1^0,\dots,z_n^0)\in\FF^n_p$ have distinct coordinates.
The Gaudin Hamiltonians  $H_s(z^0)$ of formula \Ref{Ham} define
commuting $\FF_p$-linear operators on
the $\FF_p$-vector space  $L^{\ox m}=\ox_{s=1}^nL_{m_s}$.
By formula \Ref{h invar} the Gaudin Hamiltonians preserve the $\FF_p$-subspaces
$\Sing L^{\ox m}[|m|-2k]$ and we may study eigenvectors
of the Gaudin Hamiltonians on a subspace  $\Sing L^{\ox m}[|m|-2k]$.

Consider the  {\it system of Bethe ansatz equations}
\bean
\label{BAEp}
\sum_{j\ne i}\frac 2{t_i-t_j} -\sum_{s=1}^n\frac{m_s}{t_i-z_s^0}=0,
\qquad i=1,\dots,k,
\eean
as a system of equations on $t=(t_1,\dots,t_k)\in \FF_p^k$.
If $(t_1^0,\dots,t_k^0,z^0_1,\dots,z_n^0)\in\FF_p^{k+n}$ has distinct coordinates, denote
\bean
\label{lap}
\la_s(t^0,z^0) =\sum_{l\ne s} \frac {m_{s}m_l/2}{z^0_s-z^0_l}
-\sum_{i=1}^k \frac{m_s}{z_s^0-t_i^0}\ \in \FF_p,
\qquad s=1,\dots,n.
\eean

\begin{thm} 
\label{thm BAp}
Let $p$ be a prime number and $p>|m|$. Let  $t^0\in \FF_p^{k}$ be a solution of the Bethe ansatz
equations  \Ref{BAEp}. Then the vector
 $W_{k,n,m}(t^0,z^0)$ is well-defined and lies in the subspace $\Sing L^{\ox m}[|m|-2k]$, that is,
 the equations
\bean
\label{Relp}
\sum_{s=1}^{n} (j_s + 1) (m_s - j_s) W_{J+ {\bf 1}_s}(t^0,z^0)   =  0 ,
\eean
hold, also the vector  $W_{k,n,m}(t^0,z^0)$  satisfies the equations
\bean
\label{eig}
H_s(z^0).W_{k,n,m}(t^0,z^0) = \la_s(t^0,z^0) W_{k,n,m}(t^0,z^0),
\qquad s=1,\dots, n.
\eean

\end{thm}

\begin{proof} 
The proof of Theorem \ref{thm BAp} is the same as the proof of Theorem \ref{thm BA} since identities
\Ref{id1} and \Ref{id2} hold over half integers and can be projected to $\FF_p$.
\end{proof}

\section{Two-dimensional spaces of polynomials}
\label{TD}
\subsection{Two-dimensional spaces of polynomials over $\C$}
\label{2dC}
For a function $g(x)$ denote $g' = \frac{dg}{dx}$.  For functions $g(x), h(x)$ define the {\it Wronskian}
\bea
\Wr(g(x),h(x))= g'(x)h(x)-g(x)h'(x).
\eea

\begin{thm}
 [\cite{SchV}, cf. \cite{MV2}]
\label{thm 2dC}
Let $k\in\Z_{>0}$, $m = ( m_1, \dots , m_n) \in \Z_{>0}^n$. 
Let $(t^0,z^0)\in\C^{k+n}$ have distinct coordinates.
Denote 
 \bean
 \label{yT}
 y(x)=\prod_{i=1}^k(x-t^0_i),
 \qquad
 T(x)=\prod_{s=1}^n(x-z^0_s)^{m_s}.
 \eean
We have the following two statements.

\begin{enumerate}
\item[(i)]
If $(t^0,z^0)$  
is a solution of the Bethe ansatz equations \Ref{BAE},
 then $k\leq |m|+1$, $2k\ne |m|+1$, and 
there exists a polynomial $\tilde y(x) \in \C[x]$ of degree $|m|+1-k$ such that
\bean
\label{WE}
\Wr(\tilde y(x),y(x))= T(x).
\eean

\item[(ii)]

If there exists a polynomial $\tilde y(x)$ satisfying equation
\Ref{WE}, then  $k\leq |m|+1$, $2k\ne |m|+1$  and
$(t^0,z^0)$ is a solution of the Bethe ansatz equations \Ref{BAE}.
\end{enumerate}
\end{thm}

\begin{proof} We will use the proof below in the proof of the $p$-version of Theorem \ref{thm 2dC}.
 Equation \Ref{WE} is a first order differential
equation with respect to $\tilde y(x)$. Then
\bean
\label{dern}
\left( \frac{\tilde y(x)}{y(x)} \right)' = \frac{T(x)}{y(x)^2}
\eean
and
\bean
\label{solu}
\tilde y(x) = y(x) \int\frac{T(x)}{y(x)^2}dx=y(x) \int\frac{T(x)}{\prod_{i=1}^k(x-t^0_i)^2}dx.
\eean
We have the unique presentation $T(x)= Q(x)\prod_{i=1}^k(x-t^0_i)^2 + R(x)$ with $P(x),Q(x)\in \C[x]$
such that $Q(x)=0$ if $2k>|m|$ and $Q(x) = a_{|m|-2k}x^{|m|-2k} + \dots +a_0$ is of degree $|m|-2k$ otherwise;
$\deg R(x)< 2k$. We have the unique presentation
\bean
\label{sf}
\frac{R(x)}{\prod_{i=1}^k(x-t^0_i)^2} = \sum_{i=1}^k \left(
\frac{a_{i,2}}{(x-t^0_i)^2}  + \frac{a_{i,1}}{x-t^0_i}\right),
\eean
where
\bean
\label{CO}
a_{i,2} = \frac{T(x)}{\prod_{j\ne i}^k(x-t^0_j)^2} \biggr\rvert_{x=t_i^0},
\qquad
a_{i,1} = \frac{d}{dx}\left(\frac{T(x)}{\prod_{j\ne i}^k(x-t^0_j)^2}\right) \biggr\rvert_{x=t_i^0}.
\eean
We have
\bea
\frac{d}{dx}\left(\frac{T(x)}{\prod_{j\ne i}^k(x-t^0_j)^2}\right)\biggr\rvert_{x=t_i^0}
=
\left(\sum_{s=1}^n\frac{m_s}{t_i^0-z_s^0}-\sum_{j\ne i}\frac 2{t^0_i-t_j^0}
\right)\frac{T(t^0_i)}{\prod_{j\ne i}^k(t^0_i-t^0_j)^2}.
\eea
Since $(t^0,z^0)$ has distinct coordinates we conclude that $a_{i,1}=0$ for $i=1,\dots,k$,
if and only if $(t^0,z^0)$  is a solution of \Ref{BAE}.

Let $(t^0,z^0)$ be  a solution of  \Ref{BAE}. By formula \Ref{solu} we have
\bean
\label{tif}
\phantom{aaa}
&&
\tilde y(x) 
=
 \prod_{i=1}^k(x-t^0_i)\left(c
-\sum_{i=1}^k  \frac{a_{i,2}}{x-t^0_i} \right),\quad \on{if}\  2k>|m|,
\\
\notag
&&
\tilde y(x) 
=
 \prod_{i=1}^k(x-t^0_i)\left(\frac{a_{|m|-2k}}{|m|-2k+1}x^{|m|-2k+1} + \dots +a_0x +c
-\sum_{i=1}^k  \frac{a_{i,2}}{x-t^0_i} \right),
\eean
if  $2k\leq |m|$, where $c\in\C$ is an arbitrary number.  
In each of the two cases we may choose $c$ so that $\deg \tilde y(x)\ne \deg y(x)$. Using the identity
\bean
\label{wi}
\Wr(x^\al,x^\bt)=(\al-\bt)x^{\la+\bt-1}
\eean
we obtain in this case that
\bean
\label{de}
\deg \tilde y(x) + \deg y(x) = |m|+1.
\eean
Hence $k\leq |m|+1$ and $k\ne |m|+1-k$. The first part of the theorem is proved.

Let there exist a polynomial $\tilde y(x)$ satisfying equation
\Ref{WE}. Adding to $\tilde y(x)$ the polynomial $y(x)$ with a suitable coefficient 
if necessary we may assume that 
$\deg \tilde y(x)\ne \deg y(x)$. Then \Ref{de} implies $k\leq |m|+1$ and $k\ne |m|+1-k$.

By formula \Ref{dern} we have
\bean
\label{der}
\phantom{aaa}
\Big(\frac{\tilde y(x)}{y(x)}\Big)' = a_{|m|-2k}x^{|m|-2k} + \dots +a_0 +
\sum_{i=1}^k \Big(
\frac{a_{i,2}}{(x-t^0_i)^2}  + \frac{a_{i,1}}{x-t^0_i}\Big) \ \on{if} \ 2k\leq |m|
\eean
and 
\bean
\label{der1}
\Big(\frac{\tilde y(x)}{y(x)}\Big)' = 
\sum_{i=1}^k \Big(
\frac{a_{i,2}}{(x-t^0_i)^2}  + \frac{a_{i,1}}{x-t^0_i}\Big) \ \on{if} \ 2k> |m|.
\eean
The function $\frac{\tilde y(x)}{y(x)}$ has a unique decomposition into
the sum of a polynomial and simple fractions.  The term by term derivative of that decomposition equals
the right-hand side of \Ref{der} or \Ref{der1}. Hence  all of coefficients $a_{i,1}$ must be zero. 
Hence the roots of $y(x)$ satisfy the Bethe ansatz equations.
\end{proof}

\begin{rem}This construction assigns 
to a solution $(t^0,z^0)$ of the Bethe ansatz equations
the two-dimensional
subspace $\langle \tilde y(x), y(x)\rangle$ of the space of polynomials $\C[x]$ 
such that $\deg y(x)=k, \deg\tilde y(x)=|m|-k+1$, $ \Wr(y(x),\tilde y(x))=T(x)$.
That subspace is a point of the Grassmannian of two-dimensional subspaces of
$\C[x]$.
\end{rem}

\subsection{Two-dimensional spaces of polynomials over $\FF_p$}
\label{}

\begin{thm}

\label{thm 2dF}
Let $k\in\Z_{>0}$,  $m = ( m_1, \dots , m_n) \in \Z_{>0}^n$.   Let $p>|m|+1$, $p>n+k$. 
Let $(t^0,z^0)\in\FF_p^{k+n}$ have distinct coordinates.
Denote 
 \bean
 \label{yTp}
 y(x)=\prod_{i=1}^k(x-t^0_i),
 \qquad
 T(x)=\prod_{s=1}^n(x-z^0_s)^{m_s} \ \in \FF_p[x].
 \eean
We have the following two statements.

\begin{enumerate}
\item[(i)]
If $(t^0,z^0)$  
is a solution of the Bethe ansatz equations \Ref{BAEp},
 then $k\leq |m|+1$, $2k\ne |m|+1$, and 
there exists a polynomial $\tilde y(x) \in \FF_p[x]$ of degree $|m|+1-k$ such that
\bean
\label{WEp}
\Wr(\tilde y(x),y(x))= T(x).
\eean

\item[(ii)]

If there exists a polynomial $\tilde y(x)\in\FF_p[x]$ satisfying equation
\Ref{WEp}, then  $k\leq |m|+1$, $2k\ne |m|+1$  and
$(t^0,z^0)$ is a solution of the Bethe ansatz equations \Ref{BAEp}.
\end{enumerate}
\end{thm}

\begin{proof}

\begin{lem}
\label{sfll}
 Let $p$ be a prime number.
Let $d_1,\dots,d_k\in\Z_{>0}$ with $d_i\leq 2$ for all $i$. Let  $t^0_1,\dots,t^0_k\in\FF_p$ 
be distinct and $T(x)\in\FF_p[x]$. Then
there exists a unique presentation
\bean
\label{sfr}
\frac{T(x)}{\prod_{i=1}^k(x-t^0_i)^{d_i}}  = Q(x) + \sum_{i=1}^k\sum_{j=1}^{d_i}\frac{a_{i,j}}{(x-t^0_i)^j},
\eean
where $Q(x)\in\FF_p[x]$ and
\bean
\label{aij}
a_{i,j}  = \frac{d^{j-1}}{dx^{j-1}}\left(\frac{T(x)}{\prod_{l\ne i}^k(x-t^0_l)^{d_l}}\right) \biggr\rvert_{x=t_i^0}.
\eean
\end{lem}
\begin{proof} The uniqueness is clear. Let us show the existence. Lift $t^0_1,\dots,t^0_k$, $T(x)$ to
$t^1_1,\dots,t^1_k\in\Z$, $T^1(x)\in\Z[x]$. We have 
\bean
\label{ppr}
\frac{T^1(x)}{\prod_{i=1}^k(x-t^1_i)^{d_i}}  = Q^1(x) + \sum_{i=1}^k\sum_{j=1}^{d_i}\frac{a^1_{i,j}}{(x-t^1_i)^j},
\eean
where $Q^1(x) \in \Z[x]$ and
\bean
\label{aijd}
a^1_{i,j}  = \frac{d^{j-1}}{dx^{j-1}}\left(\frac{T^1(x)}{\prod_{j\ne i}^k(x-t^1_j)^{d_j}}\right) \biggr\rvert_{x=t_i^0}.
\eean
It is easy to see that for $j=1,2$ and all $i$
the coefficient $a^1_{i,j}$ has a well-defined projection to $\FF_p$. 
By projecting \Ref{ppr} to $\FF_p$ we obtain a presentation of \Ref{sfr}.
\end{proof}

The proof of Theorem \ref{thm 2dF} is based on Lemma \ref{sfll} and is
analogous to the proof of Theorem \ref{thm 2dC}.
If $(t^0,z^0)$ is a solution of \Ref{BAEp}, then
\bea
\left(\frac{\tilde y(x)}{y(x)} \right)' = \frac{T(x)}{\prod_{i=1}^k(x-t^0)^2}=
Q(x) + \sum_{i=1}^k\frac{a_{i,2}}{(x-t^0_i)^2},
\eea
where $a_{i,2}$ are given by \Ref{aij}; $Q(x)\in\FF_p[x]$, $Q(x) = 0$ if $2k>|m|$ and
$Q(x)= a_{|m|-2k}x^{|m|-2k} + \dots +a_0$ is of degree $|m|-2k+1$ if $2k\leq |m|$, see Section \ref{2dC}.

If $2k\leq |m|$, then
\bea
\tilde y(x) 
=
 \prod_{i=1}^k(x-t^0_i)\left(\frac{a_{|m|-2k}}{|m|-2k+1}x^{|m|-2k+1} + \dots +a_0x 
-\sum_{i=1}^k  \frac{a_{i,2}}{x-t^0_i} \right) 
\eea
is a polynomial of degree $|m|-k+1$ satisfying \Ref{WE}. Notice that the polynomial in the brackets is well-defined since
$p>|m|+1$. If $2k>|m|$, then
\bea
\tilde y(x) 
=
- \prod_{i=1}^k(x-t^0_i)\left(
\sum_{i=1}^k  \frac{a_{i,2}}{x-t^0_i} \right)
\eea
is a polynomial satisfying \Ref{WEp} of degree $<k$. Formula \Ref{wi} and inequality  $p>k+n$ imply \Ref{de}.
The first part of Theorem \ref{thm 2dF} is proved.

Let there exist a polynomial $\tilde y(x)$ satisfying equation
\Ref{WEp}. Adding to $\tilde y(x)$ a suitable polynomial of the form 
$c(x^p)y(x)$ for some $c(x)\in\FF_p[x]$ if necessary,
we may assume that
$\deg\tilde y(x) - \deg y(x) \not\equiv 0$ mod $p$. Then \Ref{de} holds, $k\leq |m|+1$ and $k\ne |m|+1-k$.

By formula \Ref{dern} and Lemma \ref{sfll} we have 
\bean
\label{derp}
\phantom{aaa}
\Big(\frac{\tilde y(x)}{y(x)}\Big)' = a_{|m|-2k}x^{|m|-2k} + \dots +a_0 +
\sum_{i=1}^k \Big(
\frac{a_{i,2}}{(x-t^0_i)^2}  + \frac{a_{i,1}}{x-t^0_i}\Big) \ \on{if} \ 2k\leq |m|
\eean
and 
\bean
\label{der1p}
\Big(\frac{\tilde y(x)}{y(x)}\Big)' = 
\sum_{i=1}^k \Big(
\frac{a_{i,2}}{(x-t^0_i)^2}  + \frac{a_{i,1}}{x-t^0_i}\Big) \ \on{if} \ 2k> |m|.
\eean
The function $\frac{\tilde y(x)}{y(x)}$ has a unique decomposition into
the sum of a polynomial and simple fractions.  The term by term derivative of that decomposition equals
the right-hand side of \Ref{derp} or \Ref{der1p}. Hence  all of coefficients $a_{i,1}$ must be zero. 
Hence the roots of $y(x)$ satisfy the Bethe ansatz equations.
\end{proof}

\begin{rem} This construction assigns 
to a solution $(t^0,z^0)$ of the Bethe ansatz equations \Ref{BAEp}
the two-dimensional
subspace $\langle \tilde y(x), y(x)\rangle$ of the space of polynomials $\FF_p[x]$ 
such that $\deg y(x)=k, \deg\tilde y(x)=|m|-k+1$, $ \Wr(y(x),\tilde y(x))=T(x)$.
That subspace is a point of the Grassmannian of two-dimensional subspaces in
$\FF_p[x]$.
\end{rem}

\section{Example:  the case $k=1$ }
\label{sec ex}

\subsection{ Gaudin model on $\Sing L^{\ox m}[|m|-2]$ }

Let  $m = ( m_1, \dots , m_n) \in \Z_{>0}^n$ and $p>|m|+1$. Consider the Gaudin model on
$\Sing L^{\ox m}[|m|-2]$ over $\FF_p$ . That means that $k=1$ in the notations of the previous sections.  
A basis of $L^{\ox m}[|m|-2]$  is formed by the vectors
\bean
\label{def f}
f^{(s)} = v_{m_1}\ox\dots\ox v_{s-1}\ox fv_{m_s}\ox v_{s+1}\ox\dots\ox v_{m_n},
\quad s=1,\dots,n.
\eean
We have 
\bean
\label{def Sing}
\Sing L^{\ox m}[|m|-2] = \Big\{\sum_{s=1}^n c_s f^{(s)}\ |\ c_s\in\FF_p\ \on{and}\ \sum_{s=1}^n m_sc_s =0\Big\}.
\eean
For $s=1,\dots,n$, define the vectors   $w_s\in \Sing L^{\ox m}[|m|-2]$ by the formula
\bean
\label{ w_j}
w_s = f^{(s)}-\frac{m_s}{|m|}\sum_{l=1}^n f^{(l)}.
\eean
We have
\bean
\label{sum=0}
w_1+\dots+w_n=0.
\eean
By \cite[Lemma 4.2]{V4}, any $n-1$ of these vectors form a basis of $\Sing L^{\ox m}[|m|-2] $.

Let $z^0=(z^0_1,\dots,z^0_n)\in\FF_p^n$ have distinct coordinates. 
For $i=1,\dots,n$,  the Gaudin Hamiltonian
$H_i(z^0)$ acts on $L^{\ox m}[|m|-2]$ by the formulas:
\bean
&&
f^{(s)} \mapsto \sum_{j\ne i}\frac{m_im_j/2}{z_i^0-z_j^0} f^{(s)}
+\frac 1{z^0_i-z^0_s}(m_sf^{(i)}-m_if^{(s)}),
  \qquad s\ne i,
\\
\notag
&&
f^{(i)} \mapsto \sum_{j\ne i}\frac{m_im_j/2}{z_i^0-z_j^0} f^{(i)} + \sum_{j\ne i} \frac 1{z_i^0-z_j^0}
(m_if^{(j)} -m_j f^{(i)}) .
\eean
Hence
\bean
\label{w*w}
&&
w_s \mapsto \sum_{j\ne i}\frac{m_im_j/2}{z_i^0-z_j^0} w_s
+\frac 1{z^0_i-z^0_s}(m_sw_i-m_iw_s),
  \qquad s\ne i,
\\
\notag
&&
w_i \mapsto \sum_{j\ne i}\frac{m_im_j/2}{z_i^0-z_j^0} w_i + \sum_{j\ne i} \frac 1{z_i^0-z_j^0}
(m_iw_j -m_j w_i) .
\eean
Recall that the Bethe algebra of $\Sing L^{\ox m}[|m|-2]$ is the subalgebra of $\End(\Sing L^{\ox m}[|m|-2])$ generated
by the Gaudin Hamiltonians $H_i(z^0)$, $i=1,\dots,n$, and the identity operator. 
We denote it by $\B(z^0,m)$.

\subsection{ Bethe ansatz equation and algebra $\AA(z^0,m)$}
Let  $m = ( m_1, \dots , m_n) \in \Z_{>0}^n$ and $p>|m|+1$. Let $z^0=(z^0_1,\dots,z^0_n)$
have distinct coordinates. The Bethe ansatz equations of $\Sing L^{\ox m}[|m|-2]$ is the single equation
\bean
\label{BE1}
\frac{m_1}{t-z_1^0} + \dots +\frac{m_n}{t-z_n^0} = 0.
\eean
Write
\bean
\label{PQ}
\frac{m_1}{t-z_1^0} + \dots +\frac{m_n}{t-z_n^0} = \frac{P(t)}{\prod_{s=1}^n(t-z_s^0)},
\eean
where
\bean
\label{P-Q}
&&
P(t) = P(t,z^0,m)= \sum_{s=1}^nm_s\prod_{l\ne s}(t-z_l^0).
\eean

Let $\A_{\FF_p}$ be the affine line over $\FF_p$ with coordinate $t$.
Denote $U=\A_{\FF_p}-\{z_1^0,\dots,z_n^0\}$. Let $\mc O(U)$ be the ring of rational functions on
the affine line  $\A_{\FF_p}$ regular on $U$.
Introduce the algebra 
\bean
\label{A(z)}
\AA(z^0,m) = \mc O(U)/(P(t)),\qquad 
\dim_{\FF_p}\AA(z^0,m) = n-1.
\eean
Here $(P(t))$ is the ideal generated by $P(t)$. 
Let $u_s \in \AA(z^0,m)$, $s=1,\dots,n$,  be the 
image of $\frac{m_s}{t-z_s^0}$ in $\AA(z^0,m)$. The elements $u_s$ span 
$\AA(z^0,m)$ as a vector
space and 
\bean
\label{repi}
u_1+\dots+ u_n= 0.
\eean
We have
\bean
\label{u*u}
&&
u_iu_s = \frac 1{z^0_i-z^0_s}(m_su_i-m_iu_s),
  \qquad s\ne i,
\\
\notag
&&
u_iu_i = \sum_{j\ne i} \frac 1{z_i^0-z_j^0} (m_iu_j -m_j u_i).
\eean
For a function $g(t)\in \mc O(U)$ denote $[g(u)]$ its image in $\AA(z^0,m)$.
The elements $[1], [t], \dots,[t^{n-2}]$ form a basis of $\AA(z^0,m)$ over $\FF_p$.
The defining relation in $\AA(z^0,m)$ is $P([t])=0$. The following formulas express the elements $[t^i]$
 in terms of the elements $u_s$.

\begin{lem}
\label{lem ut} 
We have
\bean
\label{1}
[1]  &= & \frac{-1}{|m|}(z^0_1u_1+\dots+z^0_nu_n),
\\
\notag
[t]  &= & \frac{1}{|m|^2}\left(\sum_{s=1}^n z^0_sm_s\right)\left(\sum_{s=1}^n z^0_su_s\right)
+\frac{-1}{|m|}\left(\sum_{s=1}^n (z^0_s)^2u_s\right),
\\
\notag
[t^i]
&=& 
\frac{-1}{|m|}\sum_{j=1}^i\sum_{s=1}^n (z^0_s)^jm_s [t^{i-j}] 
+
\frac{-1}{|m|}\sum_{s=1}^n (z^0_s)^{i+1}u_s, \qquad
i\geq 0.
\eean
\qed
\end{lem}

These formulas are related to formulas for 
the $\widehat{sl}_2$-action on tensor products of modules dual to Verma modules,
see \cite{SV2} and in particular to formula (11) in  \cite{SV2}.

\subsection{Isomorphism of $\AA(z^0,m)$ and $\B(z^0,m)$}
Define the isomorphism of vectors spaces 
\bean
\al\  :\  \AA(z^0,m)\ \to\  \Sing L^{\ox m}[|m|-2],
\qquad
u_s\mapsto w_s, \qquad s=1,\dots,n,
\eean
in particular, we have
\bean
\label{al1}
\langle1\rangle:= \al([1]) = \frac{-1}{|m|}(z^0_1w_1+\dots+z^0_n w_n).
\eean

\begin{thm}
\label{thm AB}

The map
\bean
\label{btf}
[1]\mapsto \on{Id}, \quad  u_s \mapsto H_s(z^0) - \sum_{j\ne s} \frac{m_sm_j/2}{z^0_s-z^0_j}\on{Id},
\qquad
s=1,\dots,n,
\eean
extends to an algebra isomorphism 
\bean
\label{bt}
\bt : \AA(z^0,m) \to \B(z^0,m),
\eean
such that  $\al(gh)=\bt(g).\al(h)$ for any $g,h\in \AA(z^0,m)$. 
\end{thm}

\begin{proof}
The proof follows from comparing  \Ref{w*w} and \Ref{u*u}.
\end{proof}

\begin{rem}  Theorem \ref{thm AB}
says that the isomorphism $\al$ of vector spaces and the isomorphism $\bt$ of algebras
establish an isomorphism between
the $ \B(z^0,m)$-module $\Sing L^{\ox m}[|m|-2]$ and  the regular representation of the algebra
$\AA(z^0,m)$.

\end{rem}

\begin{exmp}
\label{ex1}
Theorem \ref{thm AB} in particular says that if $P(t)$ is irreducible then $\B(z^0,m) \cong \FF_{p^{n-1}}$,
where $\FF_{p^{n-1}}$ is the field with $p^{n-1}$ elements. 

For example, if  $n=3$, $m=(1,1,1)$, then 
$P(t,z^0) = 3t^2 -2(z^0_1+z^0_2+z^0_3)t + z^0_1z^0_2+z^0_1z^0_3+z^0_2z^0_3$.
If $p=5$, then $P(t,z^0)$ is irreducible  in $\FF_5[t]$
for all distinct $z^0_1, z^0_2, z^0_3\in\FF_5$  and $\B(z^0,m) \cong \FF_{25}$.
 
\end{exmp}

\begin{cor}  

We have
\bean
\label{dimB}
\dim_{\FF_p} \B(z^0,m) = n-1.
\eean

\end{cor}

\begin{cor}  
The operators $\bt([1])=\on{Id}$, $\bt([t^i])$, $i=1,\dots,n-2$, form a basis of 
the vector space $\B(z^0,m)$ over $\FF_p$. 
The operator 
\bean
\{t\}:=\bt([t])
&=&
\frac{1}{|m|^2}\left(\sum_{s=1}^n z^0_sm_s\right)\left(\sum_{s=1}^n z^0_s
\left(H_s(z^0) - \sum_{j\ne s} \frac{m_sm_j/2}{z^0_s-z^0_j}\on{Id}\right)
\right)
\\
\notag
&+&
\frac{-1}{|m|}\left(\sum_{s=1}^n (z^0_s)^2\left(H_s(z^0)
 - \sum_{j\ne s} \frac{m_sm_j/2}{z^0_s-z^0_j}\on{Id}\right)\right)
\eean
 generates $\B(z^0,m)$ as an algebra with defining relation $P(\{t\})=0$.
\end{cor}

\begin{cor}  

We have
\bean
\label{alw}
\left(H_s(z^0) - \sum_{j\ne s} \frac{m_sm_j/2}{z^0_s-z^0_j}\on{Id}\right).\langle1\rangle= w_s,
\qquad s=1,\dots,n.
\eean

\end{cor}

\subsection{Eigenvectors of $\B(z^0,m)$ and the polynomial $P(t)$}

The elements of the algebra $\AA(z^0,m)$  have the form $Q([t])$, where
$Q(t)\in \FF_p[t]$, $\deg Q(t)< n-1$. An element $Q([t])$ is an eigenvector of 
all multiplication operators of $\AA(z^0,m)$ if and only if $Q([t])$ is an eigenvector of
the multiplication by $[t]$. If $t^0\in\FF_p$ is the eigenvalue, then
$([t]-t^0)Q([t])=0$, that is, 
\bean
\label{er}
  (t - t^0)Q(t) = \on{const}\,P(t), \qquad \on{const}\in \FF_p.
\eean
Hence the set of eigenlines of all multiplication operators of $\AA(z^0,m)$ is in one-to-one correspondence with
the set of distinct roots of the polynomial $P(t)$, namely,  a root $t^0$ with decomposition
$(t - t^0)Q(t) = P(t)$ corresponds to the line generated by the element $Q([t])$.

\begin{cor}
\label{ceiro} 
The set of eigenlines of $\B(z^0,m)$ are in one-to-one correspondence with
the set of distinct roots of the polynomial $P(t)$, namely,  a root $t^0$ with decomposition
$(t - t^0)Q(t) = P(t)$ for some $Q(t)\in\FF_p[t]$ corresponds to the line generated by the vector
\bean
\label{qe}
\omega(t^0,z^0):= Q(\{t\}).\langle 1\rangle \ \ \in\  \Sing L^{\ox m}[|m|-2].
\eean
\qed
\end{cor}

Thus we have two ways to construct the eigenlines of $\B(z^0,m)$ from roots $t^0$ of the polynomial $P(t)$. 
The first is given by Theorem \ref{thm BA} and the eigenline is generated by the vector
\bean
\label{biv}
W_{1,n,m}(t^0,z^0)= \sum_{s=1}^n\frac1{t^0-z_s}f^{(s)}=\sum_{s=1}^n\frac 1{t^0-z_s}w_s.
\eean
The second is given by Corollary \ref{ceiro} and the eigenline is generated by the vector $Q(\{t\}).\langle 1\rangle$.

\begin{thm}
\label{thm coi}
The two eigenlines coincide, more precisely, we have
\bean
\label{coi}
W_{1,n,m}(t^0,z^0) = \on{const} Q(\{t\}).\langle 1\rangle,\qquad \on{const}\in\FF_p.
\eean

\end{thm}

\begin{proof} We need to show that $([t]-t^0)\al^{-1}(W_{1,n,m}(t^0,z^0))=0$ in $\AA(z^0,m)$. Indeed
\bea
([t]-t^0)\al^{-1}(W_{1,n,m}(t^0,z^0)) = \sum_{s=1}^n\frac{m_s}{t^0-z_s}\Big[\frac{t-t^0}{t-z_s}\Big]
=\sum_{s=1}^n\frac{m_s}{t^0-z_s}[1] - \sum_{s=1}^n\Big[\frac{m_s}{t-z_s}\Big] = 0
\eea
due to the Bethe ansatz equation \Ref{BE1} and formula \Ref{repi}.

\end{proof}

\subsection{Algebra  $\mc C(T)$}  
In this section, $p$ is a prime number, $p>n+1$.  Fix a monic polynomial
\bean
\label{sigma}
\phantom{aaa}
T(x) = x^n+\si_1x^{n-1}+\si_2x^{n-2} + \dots + \si_n \ \in \ \FF_p[x].
\eean
We consider the two-dimensional subspaces  $V\subset \FF_p[x]$ consisting of polynomials
of degree $n$ and 1 such that $\Wr(g_1(x), g_2(x)) = \on{const}  T(x)$, where 
 $g_1(x), g_2(x)$ is any basis
of $V$ and $\on{const}\in\FF_p$. 
Such a subspace $V$ has a unique basis of the form
\bean
\label{ba}
g_1(x)= x^n + a_1x^{n-1}+\dots+a_{n-2}x^2+a_0,
\qquad
g_2=x-t
\eean
with
\bean
\label{wrgg}
\Wr(g_1(x),g_2(x)) = (n-1)T(x).
\eean
Equation \Ref{wrgg} is equivalent to the system of equations
\bean
\label{wreq}
&&
(n-r-1)a_r - (n-r+1)a_{r-1}t - (n-1)\si_r=0,
\qquad  
r=1,\dots,n-1,
\\
\notag
&&
a_n - (n-1)\si_n=0,
\eean
where  $a_0=1$.  Expressing $a_1$ from the first equation in terms of $t$, then expressing $a_2$ from the 
first and second equations in terms of $t$ and so on, we can reformulate system \Ref{wreq} as the system
of equations
\bean
\label{wreq1}
&&
a_r - \frac{n-1}2 ( nt^r+(n-1)\si_1 t^{r-1} + \dots + (n-r)\si_r)=0, \qquad 
r=1,\dots, n-2,
\\
\label{wreq2}
&&
nt^{n-1}+(n-2)\si_1 t^{n-2} + \dots + 2 \si_2 t+\si_1 =0,
\\
\label{wreq3}
&&
a_n + \si_n = 0.
\eean
Notice that equation \Ref{wreq2} is the equation $\frac {dT }{d t}(t)=0$, where
$T(x)$ is defined in \Ref{sigma}.

Let $I\subset \FF_p[t, a_1,\dots,a_{n-2},a_n]$ be the ideal
generated by $n$ polynomials staying in the left-hand sides
of the equations of the system \Ref{wreq}. Define the algebra 
\bean
\label{mcc}
\mc C(T) =  \FF_p[t, a_1,\dots,a_{n-2},a_n]/ I.
\eean
Let $J \subset \FF_p[t]$ be the ideal generated by $\frac {d T}{d t}(t)$. Define the algebra
\bean
\label{mcct}
\tilde{\mc C}(T) =  \FF_p[t]/ J.
\eean

\begin{lem}
\label{lem JI} 
We have an isomorphism of algebras
\bean
\label{iCtC}
\tilde{\mc C}(T) \to \mc C(T),\quad [t] \mapsto [t].
\eean
\qed
\end{lem}

Let $m^0=(1,\dots,1)\in\Z^n_{>0}$. Let $z^0=(z^0_1,\dots,z^0_n)\in\FF_p^n$ be
a point with distinct coordinates.
The Bethe ansatz for $\Sing L^{\ox m^0}[|m^0| -2]$ has the form

\bean
\label{BAE 10}
\frac 1{t-z^0_1} + \dots +\frac 1{t-z^0_n} =\frac{R(t)}{T(t)} =0,
\eean
where 

\bean
\label{T=prod}
T(t)=\prod_{s=1}^n(t-z^0_s),
\qquad
R(t)=\frac{dT}{dt}(t).
\eean 
 Hence for this $T(x)$ we have 
 \bean
 \label{C=A}
 \tilde{\mc C}(T)=\AA(z^0, m^0).
 \eean

 \begin{cor}
 \label{cor ABC}
 For $T(t)$ and $R(t)$ as in \Ref{T=prod} we have
 \bean
 \label{ABC}
\AA(z^0,m^0)\cong \B(z^0,m^0)\cong
\mc C(T)
\eean 
and the $ \B(z^0,m)$-module $\Sing L^{\ox m^0}[|m^0|-2]$  is isomorphic to the regular representation
of the algebra $\mc C(T)$.

\end{cor}

\subsection{Wronski map}

Let $X_n$ be the affine space of all two-dimensional subspaces $V\subset \FF_p[x]$, each of which 
consists of polynomials of degree $n$ and  1. The space $X_n$ is identified with
the space of pairs of polynomials given by formula \Ref{ba}. 
Let $\FF_p[x]_{n} \subset \FF_p[x]$ be the affine subspace of monic polynomials of degree $n$.
Introduce the {\it Wronski map}
\bean
\label{WRO}
W_n : X_n \to \FF_p[x]_{n}, \quad \langle g_1(x),g_2(x) \rangle \mapsto \frac 1{n-1} \Wr(g_1(x), g_2(x)),
\eean
cf. \cite{MTV3}. 
The algebra $\mc C(T)$ is  the algebra of functions on the fiber $W^{-1}(T)$ of the Wronski map.

\begin{exmp} Let $n=3$ and $T(x) = x^3+\si_1x^2+\si_2x+\si_3$. Then 
$W_3^{-1}(T)$ consists of one point if the discriminant $\si_1^2-3\si_2$ of $\frac {dT}{dx}(x)$ 
equals zero; $W_3^{-1}(T)$ consists of two points if the discriminant is a nonzero square,
and is empty otherwise. Thus, $p^2$ points of $X_3$ have one preimage, $\frac{p-1}2p^2$ points
have two preimages, and  $\frac{p-1}2p^2$ points have none.
Cf. Example \ref{ex1}.

\end{exmp}

\end{document}